\documentclass[a4paper, 12pt]{amsart}
\usepackage{amsfonts, amsmath, amssymb, amsthm, enumitem}  
\usepackage[english]{babel}
\usepackage[utf8]{inputenc}
\usepackage[OT1]{fontenc}
\usepackage{bm}

\usepackage[margin=2.2cm]{geometry}

\usepackage{textcomp, cmap, comment}	
\usepackage{graphicx, wrapfig}
\usepackage{epigraph, xcolor, hyperref}
\usepackage{array,tabularx,tabulary,booktabs}
\usepackage{euscript, mathrsfs}

\newcounter{iii}

\newcommand{\bb}{{\mathcal B}}
\newcommand{\aaa}{{\mathcal A}}
\newcommand{\mm}{\mathcal M}

\newcommand{\ff}{\mathcal F}
\theoremstyle{plain}
\newtheorem{thm}{Theorem}

\newtheorem{lem}[thm]{Lemma}
\newtheorem{prop}[thm]{Proposition}

\newtheorem{cor}[thm]{Corollary}
\theoremstyle{definition}

\numberwithin{equation}{section}
\numberwithin{thm}{section}

\title{Structure of non-trivial intersecting families}
\author{Andrey Kupavskii}
\address{Moscow Institute of Physics and Technology; Email: {\tt kupavskii@ya.ru}.} 
\date{}

\begin{document}
\maketitle
\begin{abstract} We say that a family of $k$-subsets of an $n$-element set is {\it intersecting}, if any two of its sets intersect. In this paper, we study the structure of large intersecting families. Several years ago, Han and Kohayakawa (Proc. AMS, 2017), and then Kostochka and Mubayi (Proc. AMS, 2017) obtained certain structural results concerning large intersecting families. In this paper, we extend and generalize their results, giving them a conclusive form.
\end{abstract}

\section{Introduction}

For positive integer $n$ we use the standard notation $[n]=\{1,\ldots, n\}$ and, more generally,  $[a,b]:=\{a,a+1,\ldots, b\}$ for integers $a, b$, with the convention $[a,b] = \emptyset$ for $a>b$. For a set $X$, denote by $2^{X}$ its power set and, for an integer $k\ge 0$,  denote by ${X\choose k}$ the collection of all $k$-element subsets ({\it $k$-sets}) of $X$. A {\it family} is a collection of sets.  We call a family {\it intersecting} if any two of its sets intersect. An intersecting family is {\it trivial}, or a {\it star} if all its sets contain a given element. Otherwise, it is called {\it non-trivial}. The family of all $k$-sets containing a given element is called a {\it full star}.

One of the oldest and most famous results in extremal combinatorics is the Erd\H os--Ko--Rado theorem \cite{EKR}, which states that for $n\ge 2k$ the largest intersecting families in ${[n]\choose k}$ are full stars.  Hilton and Milner \cite{HM} found the size and structure of the largest non-trivial intersecting families of $k$-sets for $n>2k$. They have size ${n-1\choose k-1}-{n-k-1\choose k-1}+1$ and for $k\ge 5$ have an element that intersects all but one set. That is, they have the form $\mathcal H_{k}$, where for integer $2\le u\le k$
\begin{equation}\label{eqhu}\mathcal H_u:=\ \Big\{A\in {[n]\choose k}\ :\ [2,u+1]\subset A\Big\}\cup\Big\{A\in{[n]\choose k}\ :\  1\in A, [2,u+1]\cap A\ne \emptyset\Big\}.\end{equation}
We remark that $|\mathcal H_3| = |\mathcal H_2|$.

For any set $X$, family $\ff\subset 2^X$ and $i\in X$, we use the following standard notation
\begin{align*}
\ff(\bar i):=&\ \{F\ :\ i\notin F, F\in\ff\},\\
\ff(i):=&\ \{F\setminus\{i\}\ :\ i\in F, F\in\ff\}.
\end{align*}
For a family $\ff\subset 2^{[n]}$ and $i\in [n]$, the {\it diversity} $\gamma(\ff)$ of $\ff$ is $\min_{i\in [n]}|\ff(\bar i)|$. The following theorem due to Zakharov and the author is a slight strengthening of a result due to Frankl \cite{Fra1}. It provides a far-reaching generalization of the Hilton--Milner theorem.
\begin{thm}[Kupavskii and Zakharov \cite{KZ}]\label{thm1} Let $n>2k>0$ and $\ff\subset {[n]\choose k}$ be an intersecting family. If $\gamma(\ff)\ge {n-u-1\choose n-k-1}$ for some real $3\le u\le k$, then \begin{equation}\label{eq01}|\ff|\le {n-1\choose k-1}+{n-u-1\choose n-k-1}-{n-u-1\choose k-1}.\end{equation}
\end{thm}
This theorem implies the Hilton--Milner theorem if one substitutes $u=k$. It is sharp for integer values of $u$, as is witnessed by the example \eqref{eqhu}. On a high level, it provides us with an essentially sharp upper bound on $|\ff|$ in terms of its distance to the closest trivial family.

Numerous authors aimed to determine precisely, what are the largest intersecting families in ${[n]\choose k}$ with certain restrictions. One of such questions was studied by Han and Kohayakawa \cite{HK}, who determined the largest family of intersecting families that is neither contained in the Erd\H os--Ko--Rado family, nor in the Hilton--Milner family. In our terms, the question can be restated as follows: what is the largest intersecting family with $\gamma(\ff)\ge 2$? For $i\in[k]$ let us put $I_i:=[i+1,k+i]$ and
$$\mathcal J_i:=\ \{I_1,I_i\}\cup \Big\{F\in{[n]\choose k}\ :\ 1\in F, F\cap I_1\ne \emptyset, F\cap I_i\ne \emptyset\Big\}.$$
We note that $\mathcal J_i\subset{[n]\choose k}$ and that $\mathcal J_i$ is intersecting for every $i\in [k]$. Moreover, $\gamma(\mathcal J_i) = 2$ for $i>1$ and $\mathcal J_1$ is the Hilton--Milner family. It is an easy calculation to see that $|\mathcal J_i|\ge |\mathcal J_{i+1}|$ for every $k\ge 4$ and $i\in [k-1]$. More precisely, $|\mathcal J_1|> |\mathcal J_{2}|$ if $n\ge 2k+1$, and $|\mathcal J_i|> |\mathcal J_{i+1}|$ for $n\ge 2k+i-1$ and $i\ge 2$. For $n<2k+i-1$ we have $|\mathcal J_i| = |\mathcal J_{i+1}| = \ldots = |\mathcal J_{k-1}|$.\footnote{Indeed, the difference $|\mathcal J_i|-|\mathcal J_{i+1}|$ is ${n-k-2\choose k-2}-1$ for $i=1$ and ${n-k-i\choose k-1}-{n-k-i-1\choose k-1}={n-k-i-1\choose k-2}$ for $i\ge 2$. The last expression is strictly positive whenever $n\ge 2k+i-1$ and is $0$ if $2k\le n<2k+i-1$.}

\begin{thm}[Han and Kohayakawa \cite{HK}]\label{thmhk} Let $n>2k$, $k\ge 4$. Then any intersecting family $\ff$ with $\gamma(\ff)\ge 2$ satisfies
\begin{equation}\label{eqhk}|\ff|\le {n-1\choose k-1}-{n-k-1\choose k-1}-{n-k-2\choose k-2}+2,\end{equation}
moreover, for $k\ge 5$ the equality is attained only on the families isomorphic to $\mathcal J_2$.
\end{thm}
We note that Han and Kohayakawa also treated the case $k= 3$, as well as described the cases of equality for  $k=4$. These cases are more tedious and we omit them. Actually, a slightly weaker version of the theorem above (without uniqueness) is a consequence of the main result in the paper by Hilton and Milner \cite{HM}.

Given $l\ge 2$, denote by $\mathcal E_{l}$ an inclusion-maximal $k$-uniform intersecting family with $|\mathcal E_l(\bar 1)|=l$, $|\bigcap_{E\in \mathcal E_l(\bar 1)}E|=k-1$ (note that the family is defined up to isomorphism). Note that $\mathcal J_2$ is isomorphic to $\mathcal E_2$.
 The following theorem is one of the main results due to Kostochka and Mubayi \cite{KostM} (they have also treated the cases $k=3$ and $4$, which we omit here).
\begin{thm}[Kostochka and Mubayi \cite{KostM}]\label{thmko} Let $k\ge 5$ and $n=n(k)$ be sufficiently large. If $\ff\subset {[n]\choose k}$ is intersecting and $|\ff|> |\mathcal J_3|$ then  $\ff\subset \mathcal E_l$ for some $l\in \{0,\ldots, k-1, n-k\}$.
\end{thm}
It is  not difficult to see that for $k-1<l<n-k$ we have $\mathcal E_l\subset \mathcal E_{n-k}$: we have $\mathcal E_l(1) = \mathcal E_{n-k}(1)$ for this range. That is why these families are missing from the statement of the theorem.
The authors of \cite{KostM} used the Delta-systems method of Frankl.\footnote{The goal of their paper, was, in a way, to draw the attention of the researchers to this method.} 

Many results in extremal set theory are much easier to obtain once one assumes that $n$ is sufficiently large compared to $k$. (The possibility to apply the Delta-systems method being one of the reasons.) In particular, the bound on $n$ in Theorem~\ref{thmko} is doubly exponential in $k$.\footnote{Here we should note that a more efficient delta-system type approach and a more careful analysis can improve these dependencies to polynomial.}
Below, we state a result that is much more general than Theorems~\ref{thmhk} and~\ref{thmko} and, additionally, holds without any restriction on $n$.

We say that a family $\mm$ is {\it minimal w.r.t. common intersection} if, for any $M'\in \mm$, we have $|\bigcap_{M\in \mm\setminus\{M'\}}M|>|\bigcap_{M\in \mm}M|$.
The following theorem is the main result of this paper.

\begin{thm}\label{thmclass2} Assume that $n>2k\ge 8$. Consider an intersecting family $\ff\subset{[n]\choose k}$ with  $\gamma(\ff) = |\ff(\bar 1)|$.
Take a  subfamily $\mathcal M\subset \ff(\bar 1)$, which is minimal w.r.t. common intersection and such that $|\bigcap_{M\in \mm}M|= t$. Take the (unique) inclusion-maximal intersecting family $\ff'\subset {[n]\choose k}$, such that $\ff'(\bar 1) = \mm$. If $t\ge 3$
then we have
\begin{equation}\label{eqclass1} |\ff|\le |\ff'|,
\end{equation}
and, for $k\ge 5$, equality is possible if and only if $\ff$ is isomorphic to $\ff'$. If $k=4$ and $t=3$ then the equality is possible if and only if $\ff$ is isomorphic to one of $\mathcal E_{n-k},\mathcal J_2$.

Moreover, if $\ff$ is as above and  $|\bigcap_{F\in \ff(\bar 1)}F|\le t$ for some $t\ge 3$, then
\begin{equation}\label{eqclass2} |\ff|\le |\mathcal J_{k-t+1}|,
\end{equation}
and, for $k\ge 5$ and $n\ge 2k+(k-t+1)-1 = 3k-t$, equality is possible only if $\ff$ is isomorphic to $\mathcal J_{k-t+1}$. If $k\ge 5$ and $2k+1\le n<3k-t$, then equality holds only if $\ff$ is isomorphic to $\mathcal J_i$ with $i\ge k-t+1$.
\end{thm}
This theorem generalizes Theorems~\ref{thmhk} and~\ref{thmko} and gives a reasonable classification of {\it all} large intersecting families.
We also note that we cannot in general replace the condition $t\ge 3$
by $t\ge 2$.
Indeed, one can see that the family $\mathcal H_2$ (cf. \eqref{eqhu}) is larger than $\mathcal J_{k-1}$ for large $n$.\footnote{The size of $\mathcal J_{k-1}$ can be upper bounded by ${n-2\choose k-2}+{n-3\choose k-2}+2+(k-2)\big({n-4\choose k-2}-{n-k-2\choose k-2}\big)<{n-2\choose k-2}+2{n-3\choose k-2}$ for, say, $n>2k^3$.}

\begin{cor}\label{cormk} Let $n>2k\ge 10$ and $\ff\subset {[n]\choose k}$ be an intersecting family.
\begin{enumerate}
                           \item If  $|\ff|> |\mathcal J_3|$ then  $\ff\subset \mathcal E_l$ for some $l\in \{0,\ldots, k-1, n-k\}$.
                           \item Assume that  $|\ff|= |\mathcal J_3|$ and $\ff$ is not a subfamily of any $\mathcal E_\ell$. Then, if $n\ge 2k+2$, the family  $\ff$ is isomorphic to $\mathcal J_3$. If $n = 2k+1$ then $\ff$ is isomorphic to $\mathcal J_i$ for some $i\ge 3$.
                         \end{enumerate}
\end{cor}

\begin{proof}[Proof of Corollary~\ref{cormk}]
1. Fix any $\ff$ with $|\ff|>|\mathcal J_3|$  and assume that $\gamma(\ff) = |\ff(\bar 1)|$. First assume that $|\bigcap_{F\in\ff(\bar 1)} F|\le k-2$. We are in position to apply the second part of Theorem~\ref{thmclass2} to $\ff$, and get a contradiction with $|\ff|>|\mathcal J_3|$.
Therefore, $|\bigcap_{F\in \ff(\bar 1)}F|=k-1$ and thus $\ff(\bar 1)$ is isomorphic to $\mathcal E_l(\bar 1)$ for $l = |\mathcal F(\bar 1)|$.
2. The condition that $\ff$ is not a subfamily of any $\mathcal E_\ell$ guarantees that $|\bigcap_{F\in\ff(\bar 1)} F|\le k-2$. Since $|\ff| = |\mathcal J_3|$ and $k\ge 5$, we can apply the uniqueness part of the second part of Theorem~\ref{thmclass2}.
\end{proof}

The deduction of Theorem~\ref{thmhk} from Theorem~\ref{thmclass2} is straightforward: just apply the second part of Theorem~\ref{thmclass2} with $t=k-1$. The proof of Theorem~\ref{thmclass2} is given in the next section. The main tool is a combination of careful shifting (which preserves the non-trivial part of the structure of the family) and an extended version of the  bipartite switching trick, introduced in   \cite{FK1}, \cite{KZ}. One key observation that allows us to prove Theorem~\ref{thmclass2} is that this bipartite switching is possible even in situations when we know practically nothing about the structure of the family. (Previously, it was used in very structured situations, e.g., when the family is shifted.) \vskip+0.1cm

{\bf Remark. } The main result of this paper is a part of an unpublished manuscript \cite{Kup73}. It was under the review at J. Eur. Math. Soc. for almost 5 years before it was rejected (with two moderately positive `quick opinion' reviews). Then we decided to split the paper into two parts and publish it separately, partly to make the text more reader-friendly and partly because much of the results in \cite{Kup73} concerning the sharp version of Theorem~\ref{thm1} turned out to be known. (They were obtained in an unjustly overlooked paper by Goldwasser \cite{Gold}.)

In the meantime, unaware of Theorem~\ref{thmclass2}, Huang and Peng~\cite{HP} proved a result that is slightly stronger than Corollary~\ref{cormk}. Translating our notation to the notation of \cite{HP} (cf. \cite[Theorem 1.5]{HP}), the family $\mathcal J_3$ is $\mathcal K_2$, the family $\mathcal E_\ell$ is $\mathcal J_\ell$, and the family $\mathcal H_i$ is $\mathcal G_i$. Also note that what is denoted $\mathcal G_{k-1}$ in \cite{HP} coincides with $\mathcal E_{n-k}$.   In our terms, in \cite[Theorem 1.5]{HP} the authors show that the largest intersecting family which is not contained in one of $\{$full star, $\mathcal J_1,\mathcal J_2\}$, is one of $\mathcal J_3$, $\mathcal E_3$ or $\mathcal E_{n-k}$. (They also compare their sizes and determine, which one is bigger depending on the parameters, which is a routine computation.) This is essentially the content of Corollary~\ref{cormk}, except the latter does not treat the case  $k=4$. We also note that the `exceptional' case $n=2k+1$ with extra equality examples $\mathcal J_4,\ldots, \mathcal J_{k-1}$ is overlooked in \cite{HP}.

Recently, another proof of \cite[Theorem 1.5]{HP} was  obtained in \cite{WLF}.

\section{Proof of Theorem~\ref{thmclass2}}\label{sec4}
\subsection{Preliminaries}

Recall the definition of shifting. For a given pair of indices $1\le i<j\le n$ and a set $A \subset [n]$, define the {\it $(i,j)$-shift} $S_{ij}(A)$ as follows. If $i\in A$ or $j\notin A$, then $S_{ij}(A) = A$. If $j\in A, i\notin A$, then $S_{ij}(A) := (A-\{j\})\cup \{i\}$.
The  $(i,j)$-shift $S_{ij}(\mathcal A)$ of a family $\mathcal A$ is defined as follows:
$$S_{ij}(\mathcal A) := \{S_{ij}(A)\ :\  A\in \mathcal A\}\cup \{A\ :\  A,S_{ij}(A)\in \mathcal A\}.$$
Shifting preserves sizes of sets and families. Moreover, it preserves the property of a family to be intersecting. We refer tje reader to a survey of Frankl \cite{F3}. When applied directly, it does not, however, preserve different non-triviality related properties. One of the challenges that we overcome in this paper is to carefully apply shifting so that the non-trivial structure is preserved.

We will need a two-family generalization of Theorem~\ref{thm1}. We say that the families $\aaa,\bb$ are {\it cross-intersecting} if any set from $\aaa$ intersects any set from $\bb$. The following result, modulo uniqueness, appeared in  \cite{KZ}. \begin{cor}[\cite{KZ}, Theorem 1]\label{corkz}
  Let $a,b>0$, $n>a+b$. Let $\aaa\subset {[n]\choose a},\ \bb\subset {[n]\choose b}$ be a pair of cross-intersecting families.

  If $b<a$, or $b\ge a$ and $|\bb|\le {n-b-1+a\choose a-1}$, then
\begin{equation}\label{eqcreasy} |\aaa|+|\bb|\le {n\choose a}.\end{equation}
Moreover, the displayed inequality is strict unless $|\bb|=0$.

If $b<a$ and $|\bb|\ge {n-j\choose b-j}$ for integer $j\in [b]$, then
\begin{equation}\label{eqcreasy2} |\aaa|+|\bb|\le {n\choose a}+{n-j\choose b-j}-{n-j\choose a}.\end{equation}
Moreover, if the inequality on $\bb$ is strict, then the inequality in the displayed formula above is also strict, unless $b=a-1$, $j=1$, and $|\bb| = {n\choose b}$.
\end{cor}
Note that the roles of $a,b$ are interchanged, as compared to \cite[Theorem 1]{KZ}. Also, in Corollary~\ref{corkz} we write stronger upper bounds on $|\bb|$ for simplicity. The result in \cite{KZ} did not explicitly treat the equality case. The strictness part of \eqref{eqcreasy} for $b<a$ follows by comparing the bound with the bound in \eqref{eqcreasy2} for the case $j=b$ (then, the condition on $|\bb|$ is $|\bb|\ge 1$). Strictness in \eqref{eqcreasy} for $b\ge a$ follows by comparing the same bounds ($j=b+1$ and $j=b$) from \cite[Theorem 1, part 2]{KZ}. Finally, the equality case in \eqref{eqcreasy2} follows from \cite[Theorem~2.12 part 3]{Kup73}.

The following lemma is crucial for the proof of the second part of the main theorem. It allows us to compare the sizes of different $\ff$ with $\ff(\bar 1)$ being different minimal families. Given integers $m>2s$, let  $\mathcal T_2^s:=\{[s], [s+1,2s]\}$. Let $\ff_2^s\subset {[m]\choose k-1}$ stand for the largest family that is cross-intersecting with $\mathcal T_2^s$. 

\begin{lem}\label{lemmin} Let $k, s$ and $m\ge k+s$ be integers, $k\ge 4$. Given a family $\mathcal H\subset {[m]\choose s}$ with $\tau(\mathcal H)=2$ and minimal w.r.t. this property, consider  the maximal family $\ff\subset {[m]\choose k-1}$ that is cross-intersecting with $\mathcal H$. Then the unique maximum of $|\ff|+|\mathcal H|$ is attained when $\mathcal H$ is isomorphic to $\mathcal T_2^s$ and $\ff$ is  isomorphic to $\ff_2^s$.
\end{lem}
\begin{proof}
  Let us first express $|\ff_2^s|$. It is not difficult to see that
\begin{footnotesize}\begin{align}
\notag  |\ff_2^s|\  =\ & {m-1\choose k-2}-{m-s-1\choose k-2} + \\
\notag   & {m-2\choose k-2}-{m-s-2\choose k-2}+\\
\notag   &\cdots \\
\label{eqfs}   &{m-s\choose k-2}-{m-2s\choose k-2}.
\end{align}\end{footnotesize}
Indeed, in the first line we count the sets containing $1$ that intersect $[s+1,2s]$, in the second line we count the sets not containing $1$, containing $2$ and intersecting $[s+1,2s]$ etc.

Rather surprisingly, we can bound the size of $\ff$ for any $\mathcal H$ in a similar way. Suppose that $z:=|\mathcal H|$ and $\mathcal H = \{H_1,\ldots, H_z\}$. Since $\mathcal H$ is minimal, for each $l\in[z]$ there exists an element $i_l$ such that $i_l\notin H_l$ and $i_l\in \bigcap_{j\in[z]\setminus \{l\}} H_j$. (All $i_l$ are of course different.) Applying Bollobas' set-pairs inequality \cite{Bol} to $\mathcal H$ and $\{i_{l}\ :\ l\in[z]\}$, we get that $|\mathcal H|\le {s+1\choose s}=s+1$.

For each $l=2,\ldots, z$, we count the sets $F\in\ff$ such that $F\cap \{i_2,\ldots, i_l\}=\{i_l\}$. (Cf. the displayed inequality below.) Such sets must additionally intersect $H_{l}\setminus \{i_2,\ldots, i_{l-1}\}$. Note that $H_1\supset \{i_2,\ldots, i_z\}$. This covers all sets from $\ff$ that intersect $\{i_2,\ldots, i_z\}$ and gives the first $z-1$ lines in the displayed inequality below. Next, we have to deal with sets from $\ff$ that do not intersect $\{i_2,\ldots, i_z\}$. Firstly, they must intersect $H_1$.
Assuming that $H_1\setminus \{i_2,\ldots,i_z\} = \{j_1,\ldots, j_{s+1-z}\}$, for each  $l=1,\ldots, s+1-z$ we further count the sets $F\in \ff$ such that $F\cap \{i_2,\ldots, i_z, j_1,\ldots, j_l\}=\{j_l\}$. Since the intersection of $\mathcal H$ is empty, there must exist $i\ge 2$ such that $j_l\notin H_i$. Thus, these sets must additionally intersect $H_{i}\setminus \{i_2,\ldots, i_{z}\}$. Note that the set $H_{i}\setminus \{i_2,\ldots, i_{z}\}$ has size $s-z+2$. This count gives the last $s+1-z$ lines in the displayed inequality below, in which we upper bound $|H_{i}\setminus \{i_2,\ldots, i_{z}\}|$ by $s-z+2$.
 Since $F\cap H_1\ne \emptyset$ for any $F\in \ff$ and given that the classes for different $l$ are disjoint, we clearly counted each set from $\ff$ exactly once. (However, we may also have counted some sets that are not in $\ff$.) Doing this count, we get the following bound on $\ff$.
\begin{footnotesize}\begin{align}
 \notag |\ff|\ \le \ & {m-1\choose k-2}-{m-s-1\choose k-2} + \\
\notag   & {m-2\choose k-2}-{m-s-1\choose k-2}+\\
\notag   &\cdots\\
\notag   & {m-z+1\choose k-2}-{m-s-1\choose k-2}+\\
\notag   & {m-z\choose k-2}-{m-s-2\choose k-2}+\\
\notag   &\cdots \\
\label{eqfz3}   &{m-s\choose k-2}-{m-2s-2+z\choose k-2}\ =:\ f(z).
\end{align}\end{footnotesize}
Note that \eqref{eqfz3} coincides with \eqref{eqfs} when substituting $z=2$.
We have $f(z-1)-f(z)\ge {m-s-1\choose k-2}-{m-s-2\choose k-2}={m-s-2\choose k-3}> 1$ (here we use that $m\ge s+k$ and $k\ge 4$). Therefore, for any $z\ge z'$, \begin{equation*}\label{eqz'}|\mathcal H|+|\ff|\le f(z')+z',\end{equation*} and the inequality is strict unless $z=|\mathcal H|=z'$.

At the same time, we have $|\ff_2^s|+|\mathcal T_2^s|=f(2)+2$. Since, up to isomorphism, there is only one family $\mathcal H\subset {[m]\choose s}$ of size $2$ with $\tau(\mathcal H)=2$, we immediately conclude that the lemma holds.
\end{proof}

\subsection{Proof of Theorem~\ref{thmclass2}}\label{sec41}
Let us deal with the first part of the statement first. Choose $\ff$, $t$, $\mm$, and $\ff'$ satisfying the requirements of the theorem, and assume that $|\ff|\ge |\ff'|$, and $\ff$ is the largest for a given $\mm$. If $|\ff| = |\ff'|\le |\mathcal H_3| = |\mathcal H_2|$ then we can w.l.o.g. replace $\ff$ by $\mathcal H_3$ (cf. \eqref{eqhu}). Indeed, $|\cap_{H\in\mathcal H_3(\bar 1)}H|=3$ and $\mathcal H_3(\bar 1)$ contains  a copy of any minimal family $\mathcal M$ with $|\cap_{M\in\mathcal M}M|\ge 3$ as in the statement of the theorem.  If $|\ff|\ge |\ff'|\ge|\mathcal H_3|$ then, by Theorem~\ref{thm1}, $\gamma(\ff)\le
{n-4\choose k-3}$. Thus, in what follows,  we may suppose that $\gamma(\ff) = |\ff(\bar 1)|\le {n-4\choose k-3}$. We also suppose that $\bigcap_{M\in \mm}M = [2,t+1]$.

We need to show that $|\ff|\le |\ff'|$ and equality for $k\ge 5$ is only possible if $\ff = \ff'$. To do so, we shall gradually transform $\ff$ into $\ff'$. We will preserve all key properties and, moreover, after each exchange operation, the size of $\ff$ shall strictly increase if $\ff$ changes. In what follows, arguing indirectly, we assume that $\gamma(\ff)>\gamma(\ff')$.

First, for all $i=2,\ldots,t+1$ and $j>i$, we consecutively apply all the $S_{ij}$-shifts to $\ff$. Note that the family $\mm$ stays intact under these shifts, and the diversity and size of $\ff$ is not affected. Thus, we may assume that $\ff$ is invariant under these shifts. For any $j\ge 2$ and $S\subset [2,j]$, define $$\ff(S,[j]):=\big\{F\subset [j+1,n]\ :\  F\cup S\in \ff\big\}$$ and, for any family $\mathcal R\subset {[n]\choose k}$, let $\partial\mathcal R:=\bigcup_{R\in \mathcal R} {R\choose k-1}$ denote the {\it shadow} of $\mathcal R$.

An important consequence of the shifts we made is that, for any $i\in [2,t+1]$ and $S\subset [2,i-1]$, we have \begin{equation}\label{eqcard} \big|\ff(S\cup \{i\},[i])\big|\ge \big|\partial \ff(S,[i])\big|.\end{equation} Actually, we have $\ff(S\cup \{i\},[i])\supset \partial \ff(S,[i])$.  Indeed, for any $F'\in \partial \ff(S,[i])$, there exist $j\in [i+1,n]$ and $F\in \ff(S,[i])$, such that $F'\cup \{j\}=F$, and, since we performed the $(i,j)$--shift, we have $F'\cup \{i\}\cup S\in \ff$. That is, $F'\in \ff(S\cup \{i\},[i])$.

Recall that we have $|\ff(\bar 1)|\le {n-4\choose k-3}$.
\begin{prop} For any $i\in [2,4]$ and with the convention $[2,1]=\emptyset$, we have \begin{equation}\label{eqcard2} |\ff([2,i-1],[i])|\le {n-5\choose k-3}.\end{equation}
\end{prop}
\begin{proof} Assume that \eqref{eqcard2} does not hold. Let us show that
$$|\partial \ff([2,i-1],[i])|> {n-5\choose k-4}.$$
Recall that the Kruskal--Katona theorem in Lovasz' form \cite{Lov} states that if $\mathcal K$ is a family of $m$-element sets and, for some real $x\ge m$, we have $|\mathcal K|\ge {x\choose m}$, then $|\partial \mathcal K|\ge {x\choose m-1}$. In our case, $\mathcal K:=\ff([2,i-1],[i])\subset {[i+1,n]\choose k-i+2}$ and, by assumption, $|\mathcal K|>{n-5\choose k-3}$. Select $\mathcal K'\subset \mathcal K$ such that $|\mathcal K'| = {n-5\choose k-3}$. Define $x$ by the following equality: ${x\choose k-i+2}={n-5\choose k-3}$. Let us upper bound the value of $x$. We claim that $x\le n-4$ for any $i=2,3,4$. First, ${n-4\choose k-i+2}\ge {n-4\choose k-3}>{n-5\choose k-3}$ for any $i=2,3,4$, where the first inequality holds since $n\ge 2k+1$. Second, ${y\choose k-i+2}$ is an increasing function of $y$, and thus the value of $x$ defined above is at most $n-4$.

By the Kruskal--Katona theorem in Lovasz' form, we have $|\partial \ff([2,i-1],[i])|\ge |\partial \mathcal K'|\ge {x\choose k-i+1}$.
We have
\begin{align*}{x\choose k-i+1}&=\frac{k-i+2}{x+1-(k-i+2)}{x\choose k-i+2} = \frac{k-i+2}{x+1-(k-i+2)}{n-5\choose k-3}\\
&\overset{(\text{since }i\le 4)}{\ge} \frac{k-2}{x+1-(k-2)}{n-5\choose k-3}\overset{(\text{since } x\le n-4)}{\ge} \frac{k-2}{n-3-(k-2)}{n-5\choose k-3}\\
&>\frac{k-3}{n-k-1}{n-5\choose k-3}= {n-5\choose k-4}.\end{align*}
From the computation above, we have $|\partial \ff([2,i-1],[i])|> {n-5\choose k-4}$, and thus
$$\big|\ff(\bar 1)\big|\ge \big|\ff([2,i-1],[i])\big|+\big|\ff([2,i],[i])\big|\overset{\eqref{eqcard}}{\ge} \big|\ff([2,i-1],[i])\big|+\big|\partial \ff([2,i-1],[i])\big|> {n-4\choose k-3},$$
a contradiction.
\end{proof}
For each $i=2,\ldots, k$, consider the following bipartite graph $G_i$. The parts of $G_i$ are
\begin{align*}
\mathcal P_a^i:=\ &\Big\{P\ :\  P\in {[2,n]\choose k-1},\ P\cap [2,i]=\{i\}\Big\},\\
\mathcal P_b^i:=\ &\Big\{P\ :\  P\in {[2,n]\choose k},\ P\cap [2,i]=[2,i-1]\Big\},
\end{align*}
and edges connect disjoint sets. We identify $\mathcal P_a^i$ with ${[i+1,n]\choose k-2}$ and $\mathcal P_b^i$ with ${[i+1,n]\choose k-i+2}$.

For each $i=2,\ldots, t+1$ we will apply \eqref{eqcreasy} to
\begin{align*}
\aaa_i:=\ff(1)\cap \mathcal P_a^i \ \ \ \ \text{and}\ \ \ \ \
\bb_i:=\ff(\bar 1)\cap \mathcal P_b^i,
\end{align*}
and conclude that $|\aaa_i|+|\bb_i|\le {n-i\choose k-2}$, and that the inequality is strict unless $\bb_i=\emptyset$. Put $X = [i+1,n]$. In order to apply \eqref{eqcreasy}, we need, first, $n-i = |X|>k-2+k-i+2 = 2k-i$ (this holds since $n>2k$). Most importantly, we need $|\bb_i|\le {|X|-(k-i+2)-1+(k-2)\choose k-3} = {n-5\choose k-3}$ for  $i\in [2,4]$. This holds because of \eqref{eqcard2}. For $i\ge 5$ the uniformity of $\bb_i$ (thought of as a subset in ${[i+1,n]\choose k-i+2}$ is smaller than the uniformity of $\aaa_i$. 
Thus,  we may replace $\aaa_i,\ \bb_i$ with $\bb'_i:=\emptyset$ and $\aaa'_i:={X\choose k-2}$. The resulting family $\ff'_i$ is cross-intersecting since all sets in $\ff'_i(\bar 1)$ contain $[2,i]$ and sets in $\ff'_i(1)\setminus \ff(1)$ intersect $[2,i]$. Moreover, if $|\bb_i|\ne 0$ then $|\ff'_i|>|\ff|$.  Thus, by maximality of $|\ff|$ we may assume that all sets in $\ff(\bar 1)$ must contain $[2,i]$. Applying this argument for each $i=2,\ldots, t+1$, we conclude that all sets in $\ff(\bar 1)$ must contain $[2,i]$.

Put $\mm = \{M_1,\ldots, M_z\}$. Since $\mm$ is minimal, for each $M_l\in \mm$, $l\in [z]$, there is \begin{equation}\label{eqil} j_l\in \Big(\bigcap_{M\in \mm\setminus \{M_l\}}M\Big)\setminus \Big(\bigcap_{M\in \mm}M\Big).\end{equation}
We assume that $j_l = t+l+1$, $l\in[z]$. In particular, $\{j_1,\ldots, j_z\} = [t+2,t+z+1]$.

Next, for $i=t+2,\ldots, t+z+1$ we shall show by induction that all sets in $\ff(\bar 1)\setminus \{M_{i-t-1}\}$ contain $[2,i]$. This is true for $i=t+1$. Fix some $i$ and consider the bipartite graph $G_i$, defined above. We can apply \eqref{eqcreasy2} with $n-i,k-2,k-i+2, k-i+2$ playing roles of $n,a,b,j$, respectively. First, note that $b<a$. Second, we know that $|\bb_i|\ge {n-i\choose (k-i+2)-(k-i+2)}=1$ since $M_{i-t-1}\in \mathcal P_b^i$ (note that $M_l\notin \mathcal P_b^i$ for $l\ne i-t-1$, since all of them contain $i_l$ due to the definition of $i_l$). Therefore, $|\mathcal A_i|+|\mathcal B_i|\le {n-i\choose k-2}-{n-k-i\choose k-2}+1$. For $k\ge 5$, the inequality is strict unless $\bb_i=\{M_{i-t-1}\}$.  Replace $\aaa_i,\ \bb_i$ with $\bb'_i:=\{M_{i-t-1}\}$ and $\aaa'_i:=\{A\in {[i+1,n]\choose k-2}\ :\  A\cap M_{i-t-1}\ne \emptyset\}$. Let us denote the resulting family $\ff'_i$.

Let us briefly discuss the case $k=4$. If $k=4$ then each set in $\ff([2,4],[4])$ is a singleton, so $|\mathcal M|=2$ and we apply the argument above for $i=5$ only. Thus $n-5,2,1,1$ play the roles of $n,a,b,j$, respectively, and so the inequality is strict unless $|\bb_i| = {n-5\choose 1}$, i.e., $\bb_i$ contains all possible $1$-element sets. This means that $\ff$ is (a subfamily of) the exceptional family $\mathcal E_{n-k}$ for $k=4$. If $|\bb_i|<{n-5\choose 1},$ then after the replacement we will have $|\ff(\bar 1)|=2$ (i.e., $\ff$ will turn out to be isomorphic to a subfamily of $\mathcal J_{2}$).

Let us return to the analysis. The family $\ff'_i$ is cross-intersecting. Indeed, any set in $\ff'_i(1)\setminus \ff(1)$ contains $i$ and intersects $M_{i-t-1}$. At the same time, all sets in $\ff(\bar 1)\setminus \mathcal M$ contain $[2,i-1]$ by induction, and those of $\ff(\bar 1)\setminus\mathcal M$ that do not contain $i$ belong to $\mathcal P_b^i$. But they were removed after the replacement. Therefore, the only set not containing $i$ in $\ff'_i(\bar 1)$ is $M_{i-t-1}$, which shows the cross-intersecting property.  Since $|\ff|$ is maximal, we must have $\ff = \ff'_i$. Repeating this argument for  $i=t+2,t+3,\ldots, t+z+1$, we may assume that any set in $\ff(\bar 1)\setminus \mm$ must contain the set $[2, t+z+1]$.

Put $t':=t+z+1$. If $\ff(\bar 1)\setminus \mm$ is non-empty then, for each $l \in [z]$, the set $M_l\setminus [2,t']$  must be non-empty: otherwise, $t'-1>k$. If $\ff(\bar 1)\setminus \mm$ is empty then $\ff(\bar 1)=\mm$, which contradicts $\gamma(\ff)>\gamma(\ff')$. 
For each $l\in [z]$, select
one element $i_l\in M_l\cap [t'+1,n]$. Note that $i_l$ may coincide. Put $I:=\{i_l\ :\ l\in [z]\}$. Consider the bipartite graph $G(t',I)$ with parts
\begin{align*}
\mathcal P_a(t',I):=\ &\Big\{P \ :\ P\in {[2,n]\choose k-1},\ I\subset P,\  [2,t']\cap P=\emptyset\Big\},\\
\mathcal P_b(t',I):=\ &\Big\{P\ :\ P\in {[2,n]\choose k},\ [2,t']\subset P,\ I\cap P=\emptyset\Big\},
\end{align*}
and edges connecting disjoint sets. Put $Y = [t'+1,n]\setminus I$, $|Y| = n-t'-|I|$.  We identify $\mathcal P_a(t',I)$ with ${Y\choose k-|I|-1}$, $|I|+1\le z+1$, and $\mathcal P_b(t',I)$ with ${Y\choose k-t'+1}$, where $t'-1\ge t+z-1\ge z+3$. We note $|Y|=n-t'-|I|>k-|I|-1+k-t'+1$.  By the choice of $I$, we have $|\mm\cap \mathcal P_b(t',I)| =\emptyset$. 
Let
\begin{align*}
\aaa:=\ff(1)\cap \mathcal P_a(t',I) \ \ \ \ \text{and}\ \ \ \ \
\bb:=\ff(\bar 1)\cap \mathcal P_b(t',I).
\end{align*}
We have $k-|I|-1>k-t'+1$, and, therefore, we may apply  \eqref{eqcreasy} with $a:=k-|I|-1,\ b:=k-t'+1$. We have $b<a$  in this case. We conclude that $|\aaa|+|\bb|\le {|Y|\choose k-|I|-1}$, and that the inequality is strict unless $\bb=\emptyset$.
As before, replacing $\aaa$ with $\mathcal P_a(t',I)$ and $\bb$ with $\emptyset$ does not decrease the sum of sizes of the families and preserves the cross-intersecting property. Thus, by the choice of $\ff$, we must have $\bb=\emptyset$.

Repeating this for all possible choices of $I$, we arrive at the situation when any set from $\ff(\bar 1)\setminus\mm$ must intersect {\it any} such set $I$. Clearly, this is only possible for a set $F$ if $F\supset M_l\cap[t'+1,n]$ for some $l$. But then $|F|>|M_l|$, which is impossible. Thus $\ff(\bar 1) = \mm$, which contradicts our assumption $\gamma(\ff)>\gamma(\ff')$, so the proof of \eqref{eqclass1} with uniqueness is complete. Indeed, uniqueness follows from the fact that the inequalities \eqref{eqcreasy}, \eqref{eqcreasy2} are strict unless the family $\bb$ has sizes $0$ and $1$, respectively, or $k=4$ (we mentioned it at every application, and treated the case $k=4$ separately above). \\

Let us now prove the moreover part of the statement of Theorem~\ref{thmclass2}. First, if there is no family $\mathcal M\subset \ff(\bar 1)$, minimal w.r.t. common intersection and such that $|\bigcap_{M\in\mm}M|=t$, then we apply $(i,j)$-shifts to $\ff$ for $2\le i<j\le n$ until such family appears. Since common intersection of any subfamily of $\ff(\bar 1)$ may change by at most one after any shift, either we obtain the desired $\mm$ or we arrive at a shifted family $\ff(\bar 1)$ without such $\mm$. But the latter is impossible. Indeed, for a shifted intersecting family $\ff(\bar 1)$, we have $|\bigcap_{F\in \ff(\bar 1)}F|=[2,j]$ for some $j\ge 2$, and in that case $[2,j]\cup [j+2,k+2]\in \ff(\bar 1)$. But then the set $F_{j'}:=[2,j']\cup [j'+2,k+2]\in\ff(\bar 1)$, where $j\le j'\le k+1$. It is clear that the sets $F_j,\ldots, F_{k-t+j+1}$ form a (minimal) subfamily that has common intersection of size $t$.

Fix $\mathcal M\subset \ff(\bar 1)$, minimal w.r.t. common intersection, such that $|\bigcap_{M\in\mm}M|=t$. Applying the first part of Theorem~\ref{thmclass2}, we may assume that $\ff(\bar 1)=\mm$. Clearly,  the number of sets in $\ff(1)$ passing through $\bigcap_{M\in\mm}M$ is always the same, independently of the form of $\mm$. Thus, we need to analyze the sum of sizes of the family $\ff':=\big\{F\in \ff(1)\ :\  F\cap \bigcap_{M\in\mm}M=\emptyset\big\}$ and $\mm':= \big\{M\setminus \bigcap_{M\in\mm}M\ :\  M\in \mm\big\}.$ Note that $\mm'$ and $\ff'$ are cross-intersecting, moreover, $\tau(\mm')=2$ and $\mm'$ is minimal w.r.t. this property.

Now we can apply Lemma~\ref{lemmin} with $[2,n]\setminus \bigcap_{M\in\mm}M$ playing the role of $[m]$ and $k-t$ playing the role of $s$. Note that $|[2,n]\setminus \bigcap_{M\in\mm}M| = n-t-1\ge 2k-t$, and so the condition on $m$ from the lemma is satisfied. This proves \eqref{eqclass2}, moreover, we get that for $k\ge 5$ the inequality was strict unless $|\ff(\bar 1)|=2$ in the first place. But if $|\ff(\bar 1)|=2$ then $\ff$ is isomorphic to a subfamily of $\mathcal J_{i}$ for $i\ge k-t+1$, and the equality is possible only if $\ff$ is isomorphic to $\mathcal J_i$. Finally, we have seen in the introduction that if $n\ge 3k-t$ then $|\mathcal J_{k-t+1}|>|\mathcal J_i|$ for $i>k-t+1$.
This concludes the proof of Theorem~\ref{thmclass2}.

\section{Acknowledgements} We thank the anonymous referees for carefully reading the manuscript and pointing out numerous problems with the exposition. This research was supported by the grant of the Russian Science Foundation (RScF) No. 24-71-10021.


\begin{thebibliography}{111}

\bibitem{Bol} B. Bollob\'as, {\it On generalized graphs}, Acta Math. Acad. Sci. Hungar 16 (1965), 447--452.

\bibitem{EKR} P. Erd\H os, C. Ko, R. Rado, \textit{Intersection theorems for systems of finite sets}, The Quarterly Journal of Mathematics, 12 (1961), N1, 313--320.

\bibitem{Fra1} P. Frankl,  \textit{Erd\H os--Ko--Rado theorem with conditions on the maximal degree}, J. Comb. Theory Ser. A 46 (1987), N2, 252--263.

\bibitem{F3} P. Frankl, \textit{The shifting technique in extremal set theory}, Surveys in combinatorics, Lond. Math. Soc. Lecture Note Ser. 123 (1987), 81--110, Cambridge University
Press, Cambridge.

\bibitem{FK1} P. Frankl, A. Kupavskii, {\it Erd\H os--Ko--Rado theorem for $\{0, \pm 1\}$-vectors}, J. Comb. Theory Ser. A 155 (2018), 157--179.

\bibitem{Gold} J.L. Goldwasser, {\it Erd\H os--Ko--Rado with conditions on the minimum complementary degree}, J. Comb. Theory Ser. A 109 (2005), 45--62.

\bibitem{HK} J. Han, Y. Kohayakawa, {\it The maximum size of a non-trivial intersecting uniform family that is not a subfamily of the Hilton--Milner family}, Proc. Amer. Math. Soc. 145 (2017) N1, 73--87.

\bibitem{HM} A.J.W. Hilton, E.C. Milner, \textit{Some intersection theorems for systems of finite sets}, Quart. J. Math. Oxford 18 (1967), 369--384.

\bibitem{HP} Y. Huang, Y. Peng, {\it Stability of intersecting families}, European J. Combin. 115 (2024), Paper No. 103774.

\bibitem{KostM}A. Kostochka, Dhruv Mubayi, {\it The structure of large intersecting families} Proc. Amer. Math. Soc.  145 (2017), N6,  2311--2321.

\bibitem{Kup73} A. Kupavskii, {\it Structure and properties of large intersecting families}, unpublished, arXiv:1810.00920


\bibitem{KZ} A. Kupavskii, D. Zakharov, {\it Regular bipartite graphs and intersecting families}, J. Comb. Theory Ser. A 155 (2018), 180--189.

\bibitem{Lov} L. Lov\'asz, {\it Combinatorial Problems and Exercises}, 13.31, North-Holland, Amsterdam (1979).

\bibitem{WLF} Y. Wu, Y. Li, L. Feng, J. Liu, and G. Yu, {\it Stabilities of intersecting families revisited}, arXiv:2411.03674
\end{thebibliography}
\end{document}